\documentclass{amsart}
\usepackage{amscd,amsmath,amssymb,amsthm,bm,cases,color}
\usepackage[dvips]{graphicx}

\theoremstyle{definition}
\newtheorem{defn}{Definition}[section]
\newtheorem{rem}[defn]{Remark}

\theoremstyle{plain}
\newtheorem{thm}[defn]{Theorem}
\newtheorem{prop}[defn]{Proposition}
\newtheorem{lem}[defn]{Lemma}

\newtheorem{prob}[defn]{Problem}

\numberwithin{equation}{section}

\title[Right-angled Artin groups on finite subgraphs of disk graphs]{Right-angled Artin groups on finite subgraphs of disk graphs}

\author[E.~Kuno]{Erika Kuno}
\address{
(Erika Kuno)
Department of Mathematics,
Tokyo Institute of Technology,
2-12-1 Oh-okayama, Meguro-ku, Tokyo 152-8551, Japan
}
\email{kuno.e.aa@m.titech.ac.jp}
\date{\today}

\begin{document}
\maketitle
\begin{abstract}
Koberda proved that if a graph $\Gamma$ is a full subgraph of a curve graph $\mathcal{C}(S)$ of an orientable surface $S$, then the right-angled Artin group $A(\Gamma)$ on $\Gamma$ is a subgroup of the mapping class group ${\rm Mod}(S)$ of $S$. On the other hand, for a sufficiently complicated surface $S$, Kim-Koberda gave a graph $\Gamma$ which is not contained in $\mathcal{C}(S)$, but $A(\Gamma)$ is a subgroup of ${\rm Mod}(S)$. In this paper, we prove that if $\Gamma$ is a full subgraph of a disk graph $\mathcal{D}(H)$ of a handlebody $H$, then $A(\Gamma)$ is a subgroup of the handlebody group ${\rm Mod}(H)$ of $H$. Further, we show that there is a graph $\Gamma$ which is not contained in some disk graphs, but $A(\Gamma)$ is a subgroup of the corresponding handlebody groups.
\end{abstract}
%%%%%%%%%%%%%%%%%
%%%%%%%%%%%%%%%%%
%%%%%%%%%%%%%%%%%
%%%%%Introduction%%%%%%
%%%%%%%%%%%%%%%%%
%%%%%%%%%%%%%%%%%
%%%%%%%%%%%%%%%%%
\section{Introduction}
Let $H=H_{g, n}$ be an orientable 3-dimensional handlebody of genus $g$ with $n$ marked points.
We regard its boundary $\partial H$ as a compact connected orientable surface $S=S_{g, n}$ of genus $g$ with $n$ marked points.
We denote by
\begin{align*}
\xi(H)={\rm max}\{3g-3+n, 0\}
\end{align*}
the {\it complexity} of $H$, a measure which coincides with the number of components of a maximal multi-disk in $H$. 
We also define the complexity $\xi(S)$ of $S$ as $\xi(S)={\rm max}\{3g-3+n, 0\}$.
Let $\Gamma$ be a finite simplicial graph.
Through this paper, we denote by ${\it V}(\Gamma)$ and $ E(\Gamma)$ the vertex set and the edge set of $\Gamma$ respectively.
The {\it right-angled Artin group} on $\Gamma$ is defined by
\begin{align*}
A(\Gamma)=\left\langle V(\Gamma)\mid \left[v_{i}, v_{j}\right]=1 {\rm ~if~and~only~if}~\{v_{i}, v_{j} \}\in E(\Gamma) \right\rangle.
\end{align*}
For two groups $G_{1}$ and $G_{2}$, we write $G_{1}\leq G_{2}$ if there is an embedding from $G_{1}$ to $G_{2}$, that is, an injective homomorphism from $G_{1}$ to $G_{2}$.
Similarly, we write $\Lambda\leq \Gamma$ for two graphs $\Gamma$ and $\Lambda$ if $\Lambda$ is isomorphic to an induced subgraph of $\Gamma$.
We denote by ${\rm Mod}(H)$ and ${\rm Mod}(S)$ the handlebody group of $H$ and the mapping class group of $S$ respectively.

Right-angled Artin groups were introduced by Baudisch~\cite{Baudisch81}.
Recently, these groups have attracted much interest from $3$-dimensional topology and geometric group theory through the work of Haglund-Wise~\cite{Haglund-Wise08},\cite{Haglund-Wise12} on special cube complexes.
In particular, various mathematicians investigate subgroups of right-angled Artin groups or right-angled Artin subgroups of groups.
Crisp-Sageev-Sapir~\cite{Crisp-Sageev- Sapir08} studied surface subgroups of right-angled Artin groups.
Kim-Koberda~\cite{Kim-Koberda15} proved that for any tree $T$, there exists a pure braid group ${\rm PB}_{n}$ such that $A(T)$ is embedded in ${\rm PB}_{n}$.
Bridson~\cite{Bridson12} proved that the isomorphism problem for the mapping class group of a surface whose genus is sufficiently large is unsolvable by using right-angled Artin subgroup in mapping class groups.
%Bridson~\cite{Bridson12} solved the decision problems in mapping class groups by using right-angled Artin subgroup in mapping class groups.
See Koberda~\cite{Koberda13} for other researches about right-angled Artin groups and their subgroups.

On the other hand, the geometry of mapping class groups of surfaces is well understood. 
A handlebody group ${\rm Mod}(H)$ of $H$ is a subgroup of the mapping class group ${\rm Mod}(S)$ of $S$.
Hamenst\"{a}dt-Hensel~\cite{Hamenstadt-Hensel12} showed that ${\rm Mod}(H)$ is exponentially distorted in ${\rm Mod}(S)$.
Therefore, the geometric properties of handlebody groups may be different from those of mapping class groups.
Furthermore, disk graphs are not quasi-isometric to curve graphs (see Masur-Schleimer~\cite{Masur-Schleimer13}).
Our motivation of this article is whether the following three propositions are true when we change the assumptions of mapping class groups and curve graphs to handlebody groups and disk graphs.

\begin{prop}{\rm(}\cite[Theorem 1.1 and Proposition 7.16]{Koberda12}{\rm)}\label{main_Koberda12}
If $\Gamma\leq \mathcal{C}(S)$, then $A(\Gamma)\leq {\rm Mod}(S)$.
\end{prop}

\begin{prop}{\rm(}\cite[Theorem 2]{Kim-Koberda13}{\rm)}\label{main_Kim-Koberda13}
Let $S$ be an orientable surface with $\xi(S)\leq 2$.
If $A(\Gamma)\leq {\rm Mod}(S)$, then $\Gamma\leq \mathcal{C}(S)$.
\end{prop}

\begin{prop}{\rm(}\cite[Theorem 3]{Kim-Koberda13}{\rm)}\label{second_Kim-Koberda13}
Let $S$ be an orientable surface with $\xi(S)\geq 4$.
Then there exists a finite graph $\Gamma$ such that $A(\Gamma)\leq {\rm Mod}(S)$ but $\Gamma\not\leq \mathcal{C}(S)$.
\end{prop}

\begin{defn}\label{standard_embeding}
An embedding $f$ from $A(\Gamma)$ to ${\rm Mod}(H)$ is {\it standard} if $f$ satisfies the following two conditions.
\begin{itemize}
\item[(i)] The map $f$ maps each vertex of $\Gamma$ to a multi-disk twist;

\item[(ii)] For two distinct vertices $u$ and $v$ of $\Gamma$, the support of $f(u)$ is not contained in the support of $f(v)$.
\end{itemize}
\end{defn}

We first prove the following three theorems.
\begin{thm}\label{main_thm}
If $\Gamma\leq \mathcal{D}(H)$, then $A(\Gamma)\leq {\rm Mod}(H)$.
\end{thm}

\begin{thm}\label{second_thm}
Let $H$ be a handlebody with $\xi(H)=0$ or $\xi(H)=1$.
If $A(\Gamma)\leq {\rm Mod}(H)$, then $\Gamma\leq \mathcal{D}(H)$.
Let $H$ be a handlebody with $\xi(H)=2$.
If there exists a standard embedding $f\colon A(\Gamma)\rightarrow {\rm Mod}(H)$, then $\Gamma\leq \mathcal{D}(H)$.
\end{thm}

\begin{thm}\label{third_thm}
For $H=H_{0, 7}$ and $H=H_{1, 5}$, there exists a finite graph $\Gamma$ such that $A(\Gamma)\leq {\rm Mod}(H)$ but $\Gamma\not\leq \mathcal{D}(H)$.
\end{thm}
From Theorem~\ref{third_thm}, it follows that the converse of Theorem~\ref{main_thm} is generally not true.
Further, Kim-Koberda proved that having $N$-thick stars forces the converse of Proposition~\ref{main_Koberda12}.
\begin{prop}{\rm(}\cite[Theorem 5]{Kim-Koberda13}{\rm)}\label{third_Kim-Koberda13}
Suppose $S$ is a surface with $\xi(S)=N$ and $\Gamma$ is a finite graph with $N$-thick stars.
If $A(\Gamma)\leq {\rm Mod}(S)$, then $\Gamma\leq \mathcal{C}(S)$.
\end{prop}

We also prove the following:
\begin{thm}\label{fourth_thm}
Suppose $H$ is a handlebody with $\xi(H)=N$ and $\Gamma$ is a finite graph with $N$-thick stars.
If there is a standard embedding $f\colon A(\Gamma)\rightarrow {\rm Mod}(H)$, then $\Gamma\leq \mathcal{D}(H)$.
\end{thm}

Note that our all theorems also hold when we change handlebody groups to pure handlebody groups.
We also note that we cannot apply the argument of Kim-Koberda~\cite{Kim-Koberda13} for handlebody groups of high complexity handlebodies.
The methods in this paper are worthless for high complexity handlebodies.
%, and the author does not know whether the converse of Theorem~\ref{main_thm} is true for them.

\begin{prob}
When is the converse of Theorem~\ref{main_thm} true?
\end{prob}

%%%%%%%%%%%%%%%%%
%%%%%%%%%%%%%%%%%
%%%%%%%%%%%%%%%%%
%%%Preliminaries%%%%%%%%
%%%%%%%%%%%%%%%%%
%%%%%%%%%%%%%%%%%
%%%%%%%%%%%%%%%%%
\section{Preliminaries}
\subsection{Graph-theoretic terminology}
In this paper, a {\it graph} is a one-dimensional simplicial complex.
In particular, graphs have neither loops nor multi-edges.
For $X\subseteq V(\Gamma)$, the {\it subgraph} of $\Gamma$ {\it induced by} $X$ is the subgraph $\Lambda$ of $\Gamma$ defined by $V(\Lambda)=X$ and
\begin{align*}
E(\Lambda)=\{e\in E(\Gamma)\mid {\rm the~end~points~of}~e {\rm ~are~in~} X\}.
\end{align*}
In this case, we also say $\Lambda$ is an {\it induced subgraph} or a {\it full subgraph} of $\Gamma$.
A graph $\Gamma$ is $\Lambda$-{\it free} if no induced subgraphs of $\Gamma$ are isomorphic to $\Lambda$.
In particular, $\Gamma$ is {\it triangle}-{\it free} if no induced subgraphs of $\Gamma$ are triangles.
The {\it link} of $v$ in $\Gamma$ is the set of the vertices in $\Gamma$ which are adjacent to $v$, and denoted as ${\rm Link}(v)$.
The {\it star} of $v$ is the union of ${\rm Link}(v)$ and $\{v\}$, and denoted as ${\rm St}(v)$.
A {\it clique} is a subset of the vertex set which spans a complete subgraph.
By a link, a a star, or a clique, we often also mean the subgraphs induced by them.
For a positive integer $N$, we say $\Gamma$ has $N$-{\it thick stars} if each vertex $v$ of $\Gamma$ is contained in two cliques $K_{1}\cong K_{2}$ on $N$ vertices of $\Gamma$ whose intersection is exactly $v$.
Equivalently, ${\rm Link}(v)$ contains two disjoint copies of complete graphs on $N-1$ vertices of $\Gamma$ for each vertex $v$.

\subsection{Handlebodies}
A {\it handlebody} $H_{g}$ of genus $g$ is a compact orientable $3$-dimensional manifold constructed by attaching $g$ one-handles $D^{2}\times I$ to a $3$-ball, where $D^{2}$ is a $2$-disk and $I$ is an interval.
The boundary $\partial H_{g}$ of $H_{g}$ is a closed connected orientable surface $S_{g}$ of genus $g$.
A handlebody $H=H_{g, n}$ of genus $g$ with $n$ marked points is a handlebody of genus $g$, together with $n$ pairwise distinct points $p_{1}$, $p_{2}$, $\cdots$, $p_{n}$ on $\partial H_{g}$.
We regard the boundary $\partial H$ of $H$ as a compact connected orientable surface $S=S_{g, n}$ of genus $g$ with $n$ marked points.
A disk $d$ is {\it properly embedded} in $H$ if its boundary $\partial d$ is embedded in $\partial H$, and its interior is embedded in the interior of $H$.
A properly embedded disk $d$ is {\it essential} if the simple closed curve $\partial d$ is essential in $\partial H$, that is, $\partial d$ does not bound a disk in $\partial H$ or is not isotopic to a marked point on $\partial H$.
By a {\it disk} in $H$ we mean a properly embedded essential disk $(d, \partial d)\subseteq (H, \partial H)$.
A {\it disk twist} $\delta_{d}$ along a disk $d$ in $H$ is the homeomorphism defined by cutting $H$ along $d$, twisting one of the sides by $2\pi$ to the right, and gluing two sides of $d$ back to each other (see Figure~\ref{fig_disk_twist}).

\begin{figure}[h]
\includegraphics[scale=0.60]{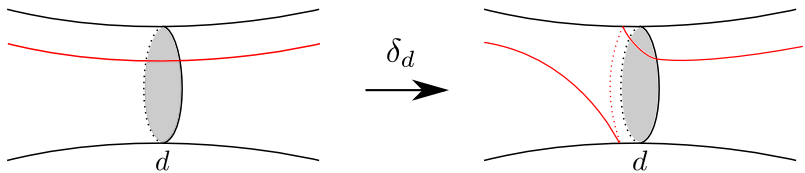}
\caption{A disk twist $\delta_{d}$ along a disk $d$ in $H$.}\label{fig_disk_twist}
\end{figure}

A {\it multi-disk} in $H$ is the union of a finite collection of disjoint disks in $H$.
The number of components of a multi-disk is at most $3g-3+n$.
A multi-disk is {\it maximal} if the number of its components is $3g-3+n$.

The {\it handlebody group} ${\rm Mod}(H)$ of $H$ is the group of orientation preserving homeomorphisms of $H$, fixing the marked points setwise, up to ambient isotopy.
The {\it mapping class group} ${\rm Mod}(S)$ of $S$ is the group defined by changing the role of homeomorphisms of $H$ into homeomorphisms of $S$ in the definition of the handlebody group.
The {\it pure handlebody group} ${\rm PMod}(H)$ of $H$ is the group of orientation preserving homeomorphisms of $H$, fixing the marked points pointwise, up to ambient isotopy.
We note that it is not important to distinguish between handlebody groups and pure handlebody groups in our considerations, since our all theorems hold for both handlebody groups and pure handlebody groups.
We call elements of ${\rm Mod}(H)$ or ${\rm Mod}(S)$ mapping classes.
An element $\Phi$ of ${\rm Mod}(H)$ is a {\it multi-disk twist} if $\Phi$ can be represented by a composition of powers of disk twists along disjoint pairwise-non-isotopic disks.

\begin{figure}[h]
\includegraphics[scale=0.50]{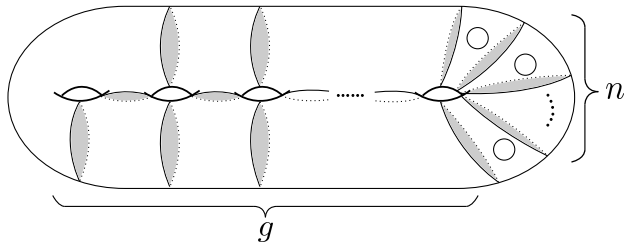}
\caption{An example of a maximal multi-disk in $H_{g, n}$.}\label{fig_multi-disk}
\end{figure}

An element $\Phi$ of ${\rm Mod}(S)$ is {\it pseudo}-{\it Anosov} if $\Phi^{n}(\alpha)\not=\alpha$ for any isotopy class $\alpha$ of simple closed curve on $S$ and $n\geq 1$ (see \cite{Birman-Lubotzky-McCarthy83}).
An element $\Phi$ of ${\rm Mod}(H)$ is {\it pseudo}-{\it Anosov} if its restriction $\Phi|_{\partial H}$ to $\partial H$ is a pseudo-Anosov element of ${\rm Mod}(\partial H)={\rm Mod}(S)$.
The {\it support} ${\rm supp}(\phi)$ of a homeomorphism $\phi$ of $H$ is defined by
\begin{align*}
{\rm supp}(\phi)=\overline{\{p\in H\mid \phi(p)\not=p\}}.
\end{align*}
Similarly, we define the support of a homeomorphism $\phi$ of $S$ as
\begin{align*}
{\rm supp}(\phi)=\overline{\{p\in S\mid \phi(p)\not=p\}}.
\end{align*}
The {\it disk graph} $\mathcal{D}(H)$ of $H$ is a graph whose vertex set is the set of isotopy classes of disks in $H$. Two vertices are adjacent if the corresponding isotopy classes admit disjoint representatives.
The {\it curve graph} $\mathcal{C}(S)$ of $S$ is a graph defined by changing the role of disks in the definition of the disk graph into properly embedded essential simple closed curves in $S$.
By a {\it curve} in $S$ we mean a properly embedded essential simple closed curve in $S$.
There exists a natural inclusion $\mathcal{D}(H)\rightarrow \mathcal{C}(S)$ given by sending an isotopy class of a disk $d$ to the isotopy class of a curve $\partial d$.
Slightly abusing the notation, we often realize isotopy classes of disks or curves as disks or curves.

%%%%%%%%%%%%%%%%%
%%%%%%%%%%%%%%%%%
%%%%%%%%%%%%%%%%%
%%%Proof of main theorem%%%
%%%%%%%%%%%%%%%%%
%%%%%%%%%%%%%%%%%
%%%%%%%%%%%%%%%%%
\section{Proof of Theorem~\ref{main_thm}}

To prove Theorem~\ref{main_thm}, it is sufficient to show the following lemma.
\begin{lem}\label{lemma_of_main_thm}
Let $\Gamma$ be a finite graph and $H$ a handlebody.
Let $i$ be an embedding from $\Gamma$ to $\mathcal{D}(H)$ as an induced subgraph.
Then for all sufficiently large $N$, the map
\begin{align*}
i_{*, N}\colon A(\Gamma)\rightarrow {\rm Mod}(H)
\end{align*}
given by sending $v$ to the $N$th power $\delta_{i(v)}^{N}$ of a disk twist $\delta_{i(v)}$ along $i(v)$ is injective.
\end{lem}

We use the following lemma to prove Lemma~\ref{lemma_of_main_thm}.
%\begin{lem}{\rm(}\cite[Theorem 1.1 and Proposition 7.16]{Koberda12}, \cite[Theorem 2 (1)]{Kim-Koberda14}, \cite[Theorem 7 (1)]{Kim-Koberda13}{\rm)}\label{lemma_of_main_thm_Kim-Koberda13}
\begin{lem}{\rm(}\cite[Theorem 7 (1)]{Kim-Koberda13}{\rm)}\label{lemma_of_main_thm_Kim-Koberda13}
Let $\Gamma$ be a finite graph and $S$ an orientable surface.
Let $i'$ be an embedding from $\Gamma$ to $\mathcal{C}(S)$ as an induced subgraph.
Then for all sufficiently large $N$, the map
\begin{align*}
i'_{*, N}\colon A(\Gamma)\rightarrow {\rm Mod}(S)
\end{align*}
given by sending $v$ to the $N$th power $T_{i'(v)}^{N}$ of a Dehn twist $T_{i'(v)}$ along $i'(v)$ in $S$ is injective.
\end{lem}

\begin{proof}[Proof of Lemma~\ref{lemma_of_main_thm_Kim-Koberda13}]
For the proof, see \cite{Koberda12}.
\end{proof}

\begin{proof}[Proof of Lemma~\ref{lemma_of_main_thm}]
Let $i\colon \Gamma\rightarrow \mathcal{D}(H)$ be an embedding as an induced subgraph.
Recall that $\mathcal{D}(H)$ is a subgraph of $\mathcal{C}(S)$, and so there is a natural embedding $j\colon \mathcal{D}(H)\rightarrow \mathcal{C}(S)$ given by sending a disk $d$ to the boundary circle $\partial d$.
We set $i'=j\circ i$.
Then $i'$ is an embedding of $\Gamma$ into $\mathcal{C}(S)$ as an induced subgraph.
By Lemma~\ref{lemma_of_main_thm_Kim-Koberda13}, there is a sufficiently large $N>0$ such that the map $i'_{*, N}\colon A(\Gamma)\rightarrow {\rm Mod}(S)$ given by sending $v\in V(\Gamma)$ to $T_{i'(v)}^{N}$ is injective.
Since $i'(v)=j\circ i(v)=\partial (i(v))$, a Dehn twist $T_{i'(v)}$ along $i'(v)$ is extended to a disk twist $\delta_{i(v)}$ along $i(v)$ in $H$.
%Since $i$ is injective, for any distinct two vertices $v$ and $w$ in $\Gamma$, $i(v)$ and $i(w)$ are distinct and non-isotopic disks in $H$.
%Hence $\delta_{i(v)}\not=\delta_{i(w)}$.
Therefore, the map $i_{*, N}\colon A(\Gamma)\rightarrow {\rm Mod}(H)$ given by sending $v\in V(\Gamma)$ to $\delta_{i(v)}^N$ is injective.
\end{proof}

From Lemma~\ref{lemma_of_main_thm}, $A(\Gamma)$ is a subgroup of ${\rm Mod}(H)$, and we have finished a proof of Theorem~\ref{main_thm}.

%%%%%%%%%%%%%%%%%
%%%%%%%%%%%%%%%%%
%%%%%%%%%%%%%%%%%
%%%Proof of second theorem%%
%%%%%%%%%%%%%%%%%
%%%%%%%%%%%%%%%%%
%%%%%%%%%%%%%%%%%
\section{Proof of Theorem~\ref{second_thm}}
The idea of the proof of Theorem~\ref{second_thm} comes from the proof of \cite[Theorem 3]{Kim-Koberda13} by changing the assumptions of mapping class groups and curve graphs to handlebody groups and disk graphs.
However, for $H=H_{1, 0}$ and $H=H_{1, 1}$ we can not apply their argument and we prove it by another way.
First we remark the following.

\begin{rem}\label{remark_of_standard_embedding}
In Definition~\ref{standard_embeding}, if supp($f$) is a maximal clique in $\mathcal{D}(H)$, then the condition (ii) implies that $v$ is an isolated vertex.
\end{rem}

\begin{proof}[Proof of Theorem~\ref{second_thm}]
First we consider the case $\xi(H)=0$, that is, $g=0$ and $n\leq 3$, or $g=1$ and $n=0$.
If $g=1$ and $n=0$, then there exists one disk in $H_{1, 0}$.
Thus, this case comes down to the case $\xi(H)=1$.
We may assume that $g=0$ and $n\leq 3$.
Then the handlebody groups are trivial.
Note that there is no essential disk in $H$.
We assume that $A(\Gamma)$ is a subgroup of ${\rm Mod}(H)$. 
Then $A(\Gamma)$ is also trivial.
Therefore $A(\Gamma)$ has no generator, and so $\Gamma$ has no vertex.
Hence $\Gamma\leq \mathcal{D}(H)$.

Suppose that $\xi(H)=1$, that is, $H=H_{0, 4}$, $H=H_{1, 0}$, or $H=H_{1, 1}$.
First, we assume $H=H_{0, 4}$.
Note that $\mathcal{D}(H)$ is an infinite union of isolated vertices.
${\rm Mod}(H)$ is virtually free, since ${\rm Mod}(S)$ is virtually free and subgroups of virtually free groups are also virtually free.
We assume that $A(\Gamma)$ is a subgroup of ${\rm Mod}(H)$. 
Then $A(\Gamma)$ is free because it is virtually free and torsion-free.
Hence, $A(\Gamma)$ has no relation, and so $\Gamma$ is a graph consists of finite isolated vertices.
Therefore $\Gamma\leq \mathcal{D}(H)$.
Secondly, we assume $H=H_{1, 0}$, or $H=H_{1, 1}$.
We note that there is only one essential disk in $H$ (see Hamenst\"{a}dt-Hensel~\cite[Section 2]{Hamenstadt-Hensel12}).
We call the disk $d$.
Hence, $\mathcal{D}(H)$ is a graph consists of a single vertex.
On the other hand, ${\rm Mod}(H)\cong \langle \delta_{d}\rangle\cong \mathbb{Z}$, where $\langle \delta_{d}\rangle$ is the group generated by a disk twist $\delta_{d}$ along $d$ (see Hamenst\"{a}dt-Hensel~\cite[Proposition 2.2]{Hamenstadt-Hensel12}).
We assume that $A(\Gamma)$ is a subgroup of ${\rm Mod}(H)$.
Then, $A(\Gamma)$ is trivial or $\mathbb{Z}$ because any non-trivial subgroup of $\mathbb{Z}$ is isomorphic to $\mathbb{Z}$.
If $A(\Gamma)$ is trivial, then $\Gamma$ has no vertex.
Hence $\Gamma\leq \mathcal{D}(H)$.
If $A(\Gamma)$ is isomorphic to $\mathbb{Z}$, then $\Gamma$ is a graph consists of a single vertex.
Therefore $\Gamma\leq \mathcal{D}(H)$.

Suppose that $\xi(H)=2$, that is, $H=H_{0, 5}$ or $H=H_{1, 2}$.
We note that $\mathcal{D}(H)$ is triangle-free.
First, we claim that the conclusion of the theorem holds for $\Gamma$ if and only if it holds for each connected component of $\Gamma$.
This is an easy consequence of the fact that $\mathcal{D}(H)$ has infinite diameter and that there exists a pseudo-Anosov homeomorphism on $H$.
So, we may assume that $\Gamma$ is connected, and so it has at least one edge.
By the hypothesis, there is a standard embedding $f$ from $A(\Gamma)$ to ${\rm Mod}(H)$.
Each vertex $v$ of $A(\Gamma)$ is mapped to a power of a single disk twist $\delta_{d}$ along $d$ by $f$, since $\Gamma$ has no isolated vertex and $\mathcal{D}(H)$ is triangle-free (see Remark~\ref{remark_of_standard_embedding}).
Hence we gain an embedding $\Gamma\rightarrow \mathcal{D}(H)$.
\end{proof}

%%%%%%%%%%%%%%%%%
%%%%%%%%%%%%%%%%%
%%%%%%%%%%%%%%%%%
%%%Proof of third theorem%%%
%%%%%%%%%%%%%%%%%
%%%%%%%%%%%%%%%%%
%%%%%%%%%%%%%%%%%
\section{Proof of Theorem~\ref{third_thm}}\label{Proof of third theorem}

Let $\Gamma_{0}$ and $\Gamma_{1}$ be the finite graphs shown in Figure~\ref{fig_Gamma_0_and_Gamma_1}.
We denote by $C_{4}$ the 4-cycle spanned by $\{a, b, c, d\}$.

\begin{figure}[h]
\includegraphics[scale=0.35]{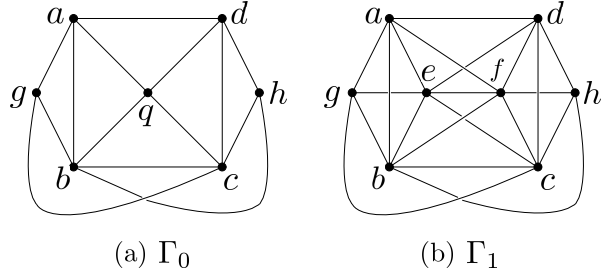}
\caption{Two graphs $\Gamma_{0}$ and $\Gamma_{1}$.}\label{fig_Gamma_0_and_Gamma_1}
\end{figure}

Let $\phi\colon A(\Gamma_{0})\rightarrow A(\Gamma_{1})$ be the map defined by $\phi(q)=ef$ and $\phi(v)=v$ for any $v\in V(\Gamma_{0})-\{q\}$.
For a graph $\Gamma$, we will denote by $\langle v\rangle$ the subgroup of $A(\Gamma)$ generated by $v\in V(\Gamma)$.

%The proof of the following lemma is similar to the proof of Kim-Koberda~\cite[Lemma 11]{Kim-Koberda13}.
\begin{lem}{\rm(}\cite[Lemma 11]{Kim-Koberda13}{\rm)}\label{first_lemma_of_third_thm_Kim-Koberda13}
The map $\phi\colon A(\Gamma_{0})\rightarrow A(\Gamma_{1})$ is injective.
\end{lem}

\begin{proof}
For the proof see the proof of Lemma~\cite[Lemma 11]{Kim-Koberda13}.
\end{proof}

\begin{figure}[h]
\includegraphics[scale=0.7]{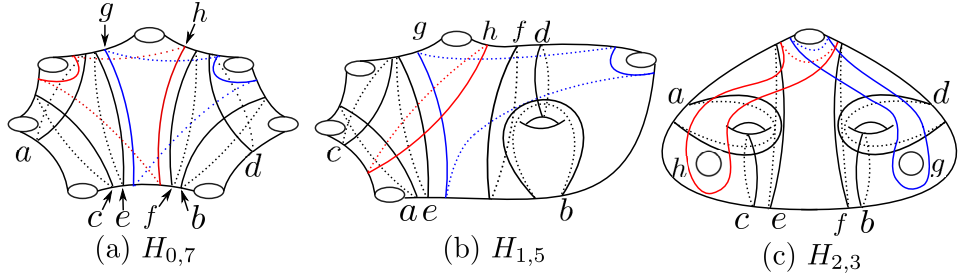}
\caption{Handlebodies $H_{0, 7}$, $H_{1, 5}$, and $H_{2, 3}$.}\label{fig_three_handlebodies}
\end{figure}

\begin{lem}\label{embedding_Gamma_{1}_into_disk_graphs}
The graph $\Gamma_{1}$ is embedded into $\mathcal{D}(H)$ if and only if $H$ is a handlebody with $\xi(H)\geq 6$, $H=H_{0, 7}$, or $H=H_{1, 5}$.
\end{lem}

We introduce the intersection number between two disks $d_{1}$ and $d_{2}$ in $H$ as follows.
\begin{align*}
i(d_{1}, d_{2})=\mid\partial d_{1}\cap \partial d_{2}\mid.
\end{align*}
Two disks $d_{1}$ and $d_{2}$ in $H$ are in {\it minimal position} if the intersection number is minimal in the isotopy classes of $d_{1}$ and $d_{2}$.
Note that two disks in $H$ intersect if and only if the boundary circles intersect in $\partial H$.
We also remark that two curves are in minimal position in $S$ if and only if they do not bound any bigons on $S$.

\begin{proof}[Proof of Lemma~\ref{embedding_Gamma_{1}_into_disk_graphs}]
We show $\Gamma_{1}$ is embedded into $\mathcal{D}(H_{0, 7})$, $\mathcal{D}(H_{1, 5})$, and $\mathcal{D}(H_{2, 3})$ as an induced subgraph.
Note that the complexities of $H_{0, 7}$, $H_{1, 5}$, and $H_{2, 3}$ are four, five, and six respectively.
We put disks $a$, $b$, $c$, $d$, $e$, $f$, $g$, and $h$ in the handlebodies as in Figure~\ref{fig_three_handlebodies} so that they form the graph $\Gamma_{1}$ in the disk graphs.
One can verify that the disks are in minimal position, since their boundary circles do not bound bigons on the boundary surfaces.
Hence $\Gamma_{1}$ is embedded into $\mathcal{D}(H_{0, 7})$, $\mathcal{D}(H_{1, 5})$, and $\mathcal{D}(H_{2, 3})$ as an induced subgraph.
Therefore $\Gamma_{1}$ is also embedded into a disk graph of a handlebody whose complexity is at least six.
We can also show that $\Gamma_{1}$ is not embedded into any other disk graphs (see Section~\ref{appendix} for the proof).
We have thus proved the lemma.
\end{proof}

Let $H$ be a handlebody with $\xi(H)=4$ or $\xi(H)=5$.
Suppose \{$a$, $b$, $c$, $d$\} are disks in $H$ which form a four cycle $C_{4}$ in $\mathcal{D}(H)$ with this order.
Let $S_{1}$ be a regular neighborhood of $\partial a$ and $\partial c$ in $\partial H$, and $S_{2}$ a regular neighborhood of $\partial b$ and $\partial d$ in $\partial H$ so that $S_{1}\cap S_{2}=\emptyset$.
Set $S_{0}=\overline{\partial H-(S_{1}\cup S_{2})}$.
Note that we regard the boundaries of $S_{0}$, $S_{1}$, and $S_{2}$ as marked points from now.

\begin{lem}\label{17_cases _of_complexity_five_handlebodies}
Let $H$ be a handlebody with $\xi(H)=5$.
Then, the triple $(S_{0}, S_{1}, S_{2})$ satisfies exactly one of the following seventeen cases, possibly after switching the roles of $S_{1}$ and $S_{2}$.

\begin{itemize}
\item[{\rm (1)}] $(S_{1}, S_{2}) \in \{(S_{0, 4}, S_{0, 6}), (S_{0, 4}, S_{1, 3}), (S_{0, 5}, S_{0, 5}), (S_{0, 5}, S_{1, 2}), (S_{1, 2}, S_{1, 2})\}$, $S_{0}\approx S_{0,2}$, and $S_{0}$ intersects both $S_{1}$ and $S_{2}$.

\item[{\rm (2)}] $(S_{1}, S_{2}) \in \{(S_{0, 4}, S_{0, 5}), (S_{0, 4}, S_{1, 2})\}$, $S_{0}\approx S_{0, 3}$, and $S_{0}$ intersects each of $S_{1}$ and $S_{2}$ at only one boundary component.

\item[{\rm (3)}] $S_{1}\approx S_{0, 4}$, $S_{2}\in \{S_{0, 5}, S_{1, 2}\}$, $S_{0}\approx S_{0, 2}\coprod S_{0, 2}$, and each component of $S_{0}$ intersects both $S_{1}$ and $S_{2}$.

\item[{\rm (4)}] $S_{1}\approx S_{0, 4}$, $S_{2}\in \{S_{0, 5}, S_{1, 2}\}$, $S_{0}\approx S_{0, 2}\coprod S_{0, 2}$, and one component of $S_{0}$ intersects each of $S_{1}$ and $S_{2}$ at only one boundary component, while the other component of $S_{0}$ intersects $S_{1}$ at just two boundary components.

\item[{\rm (5)}] $S_{1}\approx S_{0, 4}$, $S_{2}\in \{S_{0, 5}, S_{1, 2}\}$, $S_{0}\approx S_{0, 2}\coprod S_{0, 3}$ such that the $S_{0, 2}$ component intersects both $S_{1}$ and $S_{2}$ and the $S_{0, 3}$ component is disjoint from $S_{2}$, and moreover, $S_{0, 3}\cap S_{1}\approx S^{1}$.

\item[{\rm (6)}] $S_{1}, S_{2}\approx S_{0, 4}$, $S_{0}\approx S_{0, 4}$, and $S_{0}$ intersects each of $S_{1}$ and $S_{2}$ at only one boundary component.

\item[{\rm (7)}] $S_{1}, S_{2}\approx S_{0, 4}$, $S_{0}\in \{S_{0, 2}\coprod S_{0, 4}, S_{0, 2}\coprod S_{1, 1}\}$ such that the $S_{0, 2}$ component intersects both $S_{1}$ and $S_{2}$ and the $S_{0, 4}$ {\rm (}resp. $S_{1, 1}${\rm )} component is disjoint from $S_{2}$, and moreover, $S_{0, 4}\cap S_{1}\approx S^{1}$ {\rm (}resp. $S_{1, 1}\cap S_{1}\approx S^{1}${\rm )}.

\item[{\rm (8)}] $S_{1}, S_{2}\approx S_{0, 4}$, $S_{0}\approx S_{0, 3}$, and $S_{0}\cap S_{1}\approx S^{1}\coprod S^{1}$ and $S_{0}\cap S_{2}\approx S^{1}$.

\item[{\rm (9)}] $S_{1}, S_{2}\approx S_{0, 4}$, $S_{0}\approx S_{0, 2}\coprod S_{0, 3}$, and both components of $S_{0}$ intersects each of $S_{1}$ and $S_{2}$ at only one boundary component respectively.

\item[{\rm (10)}] $S_{1}, S_{2}\approx S_{0, 4}$, $S_{0}\approx S_{0, 2}\coprod S_{0, 3}$ such that the $S_{0, 2}$ component intersects both $S_{1}$ and $S_{2}$ and the $S_{0, 3}$ component is disjoint from $S_{2}$, and moreover, $S_{0, 3}\cap S_{1}\approx S^{1}$.

\item[{\rm (11)}] $S_{1}, S_{2}\approx S_{0, 4}$, $S_{0}\approx S_{0, 2}\coprod S_{0, 3}$ such that the $S_{0, 3}$ component intersects both $S_{1}$ and $S_{2}$ and the $S_{0, 2}$ component is disjoint from $S_{2}$, and moreover, $S_{0, 2}\cap S_{1}\approx S^{1}\coprod S^{1}$.

\item[{\rm (12)}] $S_{1}, S_{2}\approx S_{0, 4}$, $S_{0}\approx S_{0, 2}\coprod S_{0, 2}\coprod S_{0, 2}$, and three components of $S_{0}$ intersect both $S_{1}$ and $S_{2}$.

\item[{\rm (13)}] $S_{1}, S_{2}\approx S_{0, 4}$, $S_{0}\approx S_{0, 2}\coprod S_{0, 2}\coprod S_{0, 2}$, and two components of $S_{0}$ intersect both $S_{1}$ and $S_{2}$, while the other component of $S_{0}$ is disjoint from $S_{2}$ and intersects $S_{1}$ at just two boundary components.

\item[{\rm (14)}] $S_{1}, S_{2}\approx S_{0, 4}$, $S_{0}\approx S_{0, 2}\coprod S_{0, 2}\coprod S_{0, 2}$ such that one component of $S_{0}$ intersects both $S_{1}$ and $S_{2}$, one {\rm (}named $I${\rm )} of the other components is disjoint from $S_{2}$, the other component {\rm (}named $J${\rm )} is disjoint from $S_{1}$, and $I\cap S_{1}\approx S^{1}\coprod S^{1}$ and $J\cap S_{2}\approx S^{1}\coprod S^{1}$.

\item[{\rm (15)}] $S_{1}, S_{2}\approx S_{0, 4}$, $S_{0}\approx S_{0, 2}\coprod S_{0, 2}\coprod S_{0, 3}$ such that two $S_{0, 2}$ components intersect both $S_{1}$ and $S_{2}$ and the $S_{0, 3}$ component is disjoint from $S_{2}$, and moreover, $S_{0, 3}\cap S_{1}\approx S^{1}$.

\item[{\rm (16)}] $S_{1}, S_{2}\approx S_{0, 4}$, $S_{0}\approx S_{0, 2}\coprod S_{0, 2}\coprod S_{0, 3}$, such that one $S_{0, 2}$ component intersects both $S_{1}$ and $S_{2}$, the other $S_{0, 2}$ component {\rm (}named $I${\rm )} is disjoint from $S_{2}$, the $S_{0, 3}$ component is disjoint from $S_{1}$ or $S_{2}$ {\rm (}here we suppose that $S_{0, 3}$ is disjoint from $S_{2}${\rm )}, and moreover, $I\cap S_{1}\approx S^{1}\coprod S^{1}$ and $S_{0, 3}\cap S_{1}\approx S^{1}$.

\item[{\rm (17)}] $S_{1}, S_{2}\approx S_{0, 4}$, $S_{0}\approx S_{0, 2}\coprod S_{0, 3}\coprod S_{0, 3}$ such that the $S_{0, 2}$ component intersects both $S_{1}$ and $S_{2}$ and each $S_{0, 3}$ component is disjoint from $S_{1}$ or $S_{2}$ respectively {\rm (}here we suppose that both $S_{0, 3}$ components are disjoint from $S_{2}${\rm )}, and moreover, each $S_{0, 3}$ component intersects $S_{1}$ at only one boundary component.
\end{itemize}
\end{lem}

\begin{proof}[Proof of Lemma~\ref{17_cases _of_complexity_five_handlebodies}]
Let $\alpha$ be the number of free isotopy classes of boundary components of $S_{0}$ that are contained in $S_{1}\cup S_{2}$.
We have $\alpha>0$, since $S$ is connected and $S_{1}\cap S_{2}=\emptyset$.
Let $\xi(S_{0})$ be the sum of complexities of the components of $S_{0}$.
Then $\xi(\partial H)=\xi(S_{1})+\xi(S_{2})+\xi(S_{0})+\alpha$.
Since $S_{i}$ ($i=1, 2$) contains at least one curve, we have $2\leq \xi(S_{1})+\xi(S_{2})$.
Further, $\xi(S_{1})+\xi(S_{2})=5-\xi(S_{0})-\alpha\leq 5-0-1=4$.
Therefore it follows that $2\leq \xi(S_{1})+\xi(S_{2})\leq 4$.
We note that if $S_{1}$ or $S_{2}$ is a surface whose complexity is one, then it is homeomorphic to only $S_{0, 4}$.
In fact, if it is homeomorphic to $S_{1, 1}$, then it cannot have two curves $\partial a$ and $\partial c$ since $H_{1, 1}$ has only one isotopy class of disk.

We suppose that $\xi(S_{1})+\xi(S_{2})=4$.
Then we have $\xi(S_{0})+\alpha=1$.
If $\xi(S_{1})=1$ and $\xi(S_{2})=3$, then $S_{1}\approx S_{0, 4}$ and $S_{2}\approx S_{0, 6}, S_{1, 3}$.
If $\xi(S_{1})=2$ and $\xi(S_{2})=2$, then $S_{1}\approx S_{0,5}, S_{1, 2}$ and $S_{2}\approx S_{0,5}, S_{1, 2}$.
By the assumption that $\alpha\geq 1$, we have $\alpha=1$ and $\xi(S_{0})=0$.
Since $S_{0}$ has at least two boundary components, $S_{0}\approx S_{0, 2}$ or $S_{0}\approx S_{0, 3}$.
If $S_{0}\approx S_{0, 3}$, then this contradicts the assumption that $\alpha=1$.
Hence, we have $S_{0}\approx S_{0, 2}$.
Case (1) is immediate.

We suppose that $\xi(S_{1})+\xi(S_{2})=3$.
Then we have $\xi(S_{0})+\alpha=2$.
Without loss of generality we may assume $\xi(S_{1})=1$ and $\xi(S_{2})=2$.
It follows that $S_{1}\approx S_{0, 4}$ and $S_{2}\approx S_{0, 5}, S_{1, 2}$.
If $\alpha=1$, then $S_{0}$ forced to be an annulus and we have a contradiction of the fact that $\xi(S_{0})+\alpha=2$.
So we have $\alpha=2$ and $\xi(S_{0})=0$.
If $S_{0}$ is connected, then $\alpha=2$ implies that $S_{0}\approx S_{0, 3}$, and hence Case (2) follows.
We assume that $S_{0}$ is not connected.
By the assumption that $\alpha=2$, the number of connected components of $S_{0}$ is at most two.
We have $S_{0}\approx S_{0, 2}\coprod S_{0, 2}$ or $S_{0}\approx S_{0, 2}\coprod S_{0, 3}$.
If $S_{0}\approx S_{0, 2}\coprod S_{0, 2}$, then we have two cases where each component intersects both $S_{1}$ and $S_{2}$, and one (we name it $I$) of the component is disjoint from $S_{2}$. 
In the former case, we obtain Case (3).
In the latter case, if $S_{1}\cap I\approx S^{1}$, then there is no essential disk in $I$, and so we have a contradiction to $\alpha=2$.
Hence Case (4) follows.
If $S_{0}\approx S_{0, 2}\coprod S_{0, 3}$, then $S_{0, 3}$ has to be disjoint from $S_{2}$ and $S_{0, 3}\cap S_{1}\approx S^{1}$ since $\alpha=2$.
Then we obtain Case (5).

\begin{figure}[h]
\includegraphics[scale=0.37]{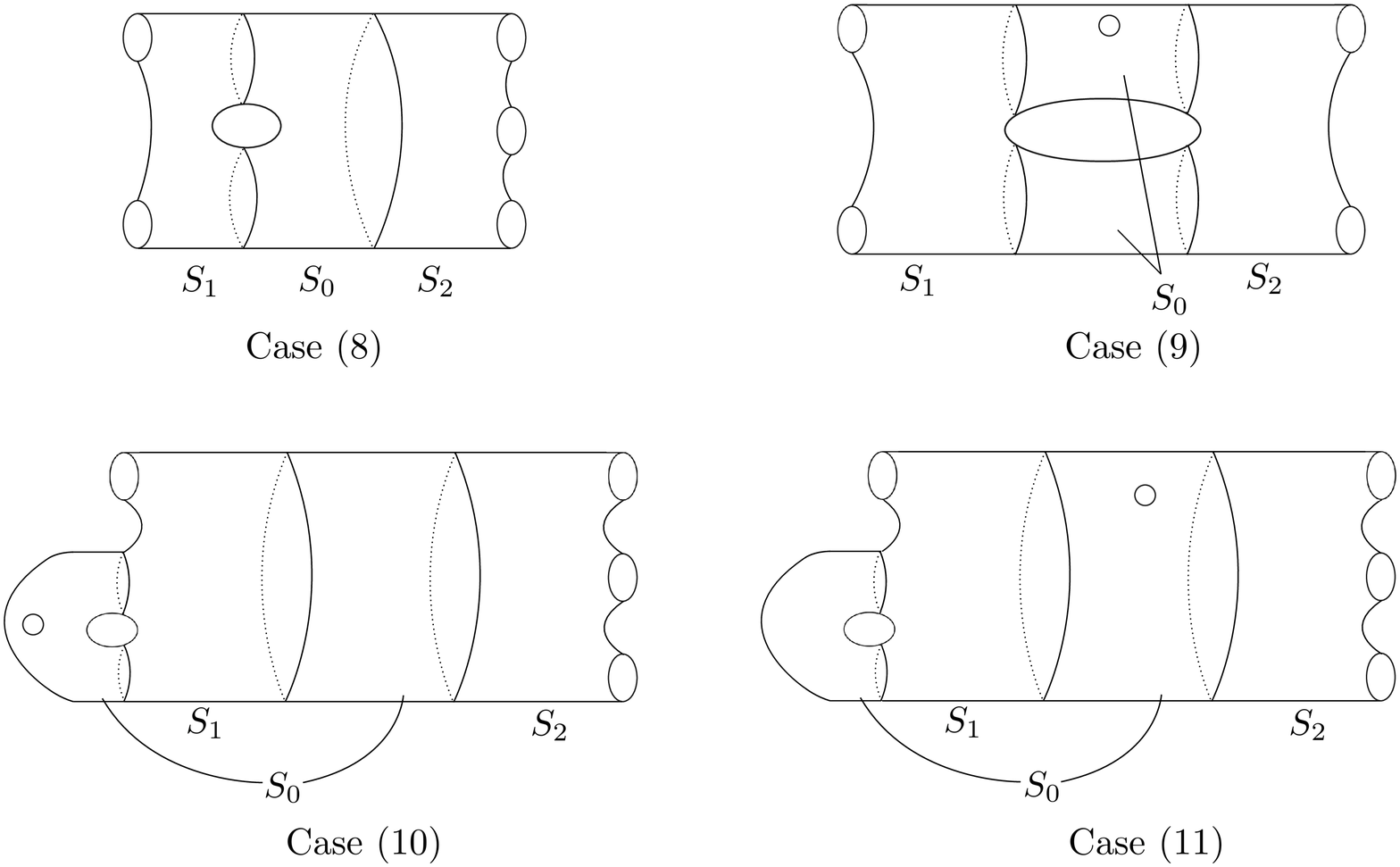}
\caption{The handlebody of Cases (8), (9), (10), and (11).}\label{fig_case_8_9_10_11}
\end{figure}

We suppose that $\xi(S_{1})+\xi(S_{2})=2$.
Then $\xi(S_{1})=1$ and $\xi(S_{2})=1$.
Thus $S_{1}\approx S_{0, 4}$ and $S_{2}\approx S_{0, 4}$.
Moreover it follows that $\xi(S_{0})+\alpha=3$.
If $\alpha=1$, then by a similar argument to that of Case (1) we have $S_{0}\approx S_{0, 2}$, and this contradicts the fact that $\xi(S_{0})=2$.
First, we consider the case where $\alpha=2$ and $\xi(S_{0})=1$.
If $S_{0}$ is connected, then $\alpha=2$ implies that $S_{0}\approx S_{0, 4}$ and $S_{0, 4}\cap S_{1}\approx S^{1}$, $S_{0, 4}\cap S_{2}\approx S^{1}$, hence Case (6) follows.
We assume that $S_{0}$ is not connected.
By the assumption that $\alpha=2$, the number of connected components of $S_{0}$ is at most two and the component which intersect both $S_{1}$ and $S_{2}$ has to be $S_{0, 2}$.
The other component (we name $I$) which is disjoint from $S_{2}$ is $S_{0, 4}$ or $S_{1, 1}$, and $I\cap S_{1}\approx S^{1}$. 
Then Case (7) follows.
Next, we consider the case where $\alpha=3$ and $\xi(S_{0})=0$.
If $S_{0}$ is connected, then $\alpha=3$ implies that $S_{0}\approx S_{0, 3}$, $S_{0, 3}\cap S_{1}\approx S^{1}\coprod S^{1}$, and $S_{0, 3}\cap S_{2}\approx S^{1}$.
Hence Case (8) follows (see Figure~\ref{fig_case_8_9_10_11}).
We assume that $S_{0}$ is not connected.
By the assumption that $\alpha=3$, the number of connected components of $S_{0}$ is at most three.
First we suppose that the number of connected components of $S_{0}$ is two.
By the assumption that $\xi(S_{0})=0$, $S_{0}\approx S_{0, 2}\coprod S_{0, 2}$ or $S_{0}\approx S_{0, 2}\coprod S_{0, 3}$.
If $S_{0}\approx S_{0, 2}\coprod S_{0, 2}$, then this contradicts the assumption that $\alpha=3$, and so we have $S_{0}\approx S_{0, 2}\coprod S_{0, 3}$.
If each component of $S_{0}$ intersects both $S_{1}$ and $S_{2}$, then Case (9) is immediate (see Figure~\ref{fig_case_8_9_10_11}).
If the $S_{0, 2}$ component intersects both $S_{1}$ and $S_{2}$, and $S_{0, 3}$ component is disjoint from $S_{2}$, then Case (10) is immediate (see Figure~\ref{fig_case_8_9_10_11}).
If the $S_{0, 3}$ component intersects both $S_{1}$ and $S_{2}$, and $S_{0, 2}$ component is disjoint from $S_{2}$, then Case (11) is immediate (see Figure~\ref{fig_case_8_9_10_11}).

\newpage
\begin{figure}[h]
\includegraphics[scale=0.37]{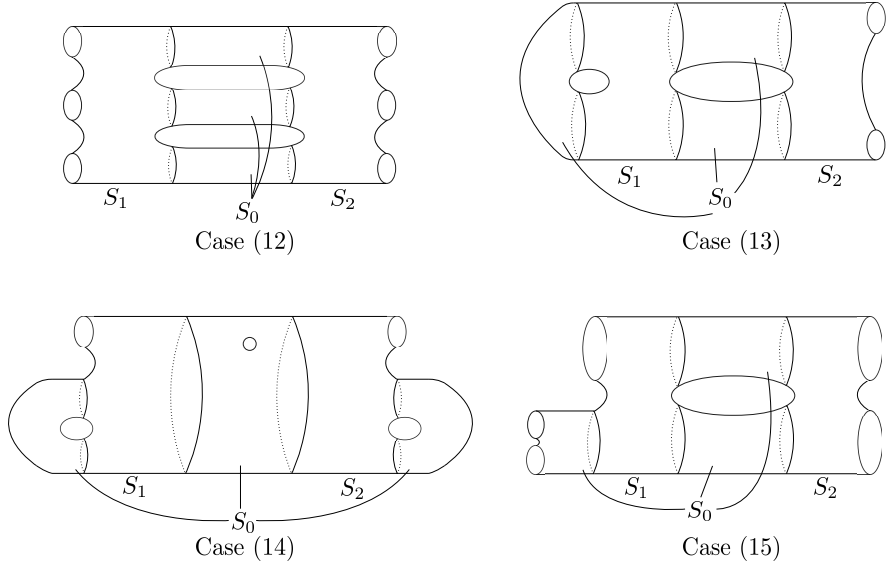}
\caption{The handlebody of Cases (12), (13), (14), and (15).}\label{fig_case_12_13_14_15}
\end{figure}

\begin{figure}[h]
\includegraphics[scale=0.37]{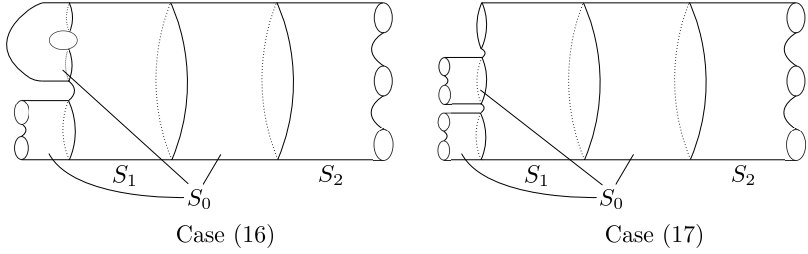}
\caption{The handlebody of Cases (16) and (17).}\label{fig_case_16_17}
\end{figure}
Finally, we suppose that the number of connected components of $S_{0}$ is three.
By the assumption that $\xi(S_{0})=0$, $S_{0}\approx S_{0, 2}\coprod S_{0, 2}\coprod S_{0, 2}$, $S_{0}\approx S_{0, 2}\coprod S_{0, 2}\coprod S_{0, 3}$, or $S_{0}\approx S_{0, 2}\coprod S_{0, 3}\coprod S_{0, 3}$.
We note that if $S_{0}\approx S_{0, 3}\coprod S_{0, 3}\coprod S_{0, 3}$, then this contradicts the assumption that $\alpha=3$.
We assume that $S_{0}\approx S_{0, 2}\coprod S_{0, 2}\coprod S_{0, 2}$.
If each component of $S_{0}$ intersects both $S_{1}$ and $S_{2}$, then Case (12) is immediate (see Figure~\ref{fig_case_12_13_14_15}).
If two components of $S_{0}$ intersect both $S_{1}$ and $S_{2}$ and the other component (we name it $I$) is disjoint from $S_{2}$, then $I\cap S_{1}\approx S^{1}\coprod S^{1}$.
We obtain Case (13) (see Figure~\ref{fig_case_12_13_14_15}).
We assume that just one component of $S_{0}$ intersects both $S_{1}$ and $S_{2}$.
We also assume that one of the other component (we name it $I$) is disjoint from $S_{2}$ and the other component (we name it $J$) is disjoint from $S_{1}$.
Then $I\cap S_{1}\approx S^{1}\coprod S^{1}$ and $J\cap S_{2}\approx S^{1}\coprod S^{1}$, and so Case (14) follows (see Figure~\ref{fig_case_12_13_14_15}).
We assume that $S_{0}\approx S_{0, 2}\coprod S_{0, 2}\coprod S_{0, 3}$.
Note that the $S_{0, 3}$ component of $S_{0}$ must not intersect both $S_{1}$ and $S_{2}$ since $\alpha=3$.
If each $S_{0, 2}$ component intersects both $S_{1}$ and $S_{2}$, then the $S_{0, 3}$ component is disjoint from $S_{2}$ and $S_{0, 3}\cap S_{1}\approx S^{1}$, and Case (15) is immediate (see Figure~\ref{fig_case_12_13_14_15}).
We assume that just one of the $S_{0, 2}$ component intersects both $S_{1}$ and $S_{2}$, and the other component (we name it $I$) is disjoint from $S_{2}$.
We also suppose that the $S_{0, 3}$ component is disjoint from $S_{2}$.
Then $I\cap S_{1}\approx S^{1}\coprod S^{1}$ and $S_{0, 3}\cap S_{1}\approx S^{1}$, hence Case (16) follows (see Figure~\ref{fig_case_16_17}).
We assume that $S_{0}\approx S_{0, 2}\coprod S_{0, 3}\coprod S_{0, 3}$.
Note that each of the $S_{0, 3}$ components of $S_{0}$ must not intersect both $S_{1}$ and $S_{2}$ since $\alpha=3$.
Then the $S_{0, 2}$ component intersects both $S_{1}$ and $S_{2}$ and two $S_{0, 3}$ components are disjoint from $S_{2}$, and moreover each $S_{0, 3}$ component intersects $S_{1}$ at only one boundary component respectively.
We obtain Case (17) (see Figure~\ref{fig_case_16_17}).
\end{proof}

\begin{lem}\label{five_cases_of complexity_four_handlebodies}
Let $H$ be a handlebody with $\xi(H)=4$.
Then, the triple $(S_{0}, S_{1}, S_{2})$ satisfies exactly one of the following five cases, possibly after switching the roles of $S_{1}$ and $S_{2}$.

\begin{itemize}
\item[{\rm (1)$'$}] $S_{1}\in \{S_{1, 2}, S_{0, 5}\}$, $S_{2}\approx S_{0, 4}$, $S_{0}\approx S_{0,2}$, and $S_{0}$ intersects both $S_{1}$ and $S_{2}$.

\item[{\rm (2)$'$}] $S_{1}, S_{2}\approx S_{0, 4}$, $S_{0}\approx S_{0, 3}$, and $S_{0}$ intersects each of $S_{1}$ and $S_{2}$ at only one boundary component.

\item[{\rm (3)$'$}] $S_{1}, S_{2}\approx S_{0, 4}$, $S_{0}\approx S_{0, 2}\coprod S_{0, 2}$, and each component of $S_{0}$ intersects both $S_{1}$ and $S_{2}$.

\item[{\rm (4)$'$}] $S_{1}, S_{2}\approx S_{0, 4}$, $S_{0}\approx S_{0, 2}\coprod S_{0, 2}$, and one component of $S_{0}$ intersects each of $S_{1}$ and $S_{2}$ at only one boundary component, while the other component of $S_{0}$ intersects $S_{1}$ at two boundary components.

\item[{\rm (5)$'$}] $S_{1}, S_{2}\approx S_{0, 4}$, $S_{0}\approx S_{0, 2}\coprod S_{0, 3}$ such that the $S_{0, 2}$ component intersects both $S_{1}$ and $S_{2}$ and the $S_{0, 3}$ component is disjoint from $S_{2}$, and moreover, $S_{0, 3}\cap S_{1}\approx S^{1}$.
\end{itemize}
\end{lem}

Note that we obtain Lemma~\ref{five_cases_of complexity_four_handlebodies} by changing the assumption of a surface $S$ in \cite[Lemma 13]{Kim-Koberda13} to a handlebody $H$.
We can prove this by the same process as the proof of Lemma~\ref{17_cases _of_complexity_five_handlebodies}.
In our case, if $S_{1}$ and $S_{2}$ are surfaces whose complexities are one, then they are homeomorphic to only $S_{0, 4}$.
On the other hand, in \cite[Lemma 13]{Kim-Koberda13} if $S_{1}$ and $S_{2}$ are surfaces whose complexities are one, then they are homeomorphic to $S_{0, 4}$ or $S_{1, 1}$ .

%\begin{proof}[Proof of Lemma~\ref{five_cases_of complexity_four_handlebodies}]
%We can prove this by the same process as the proof of Lemma~\ref{17_cases _of_complexity_five_handlebodies}.
%\end{proof}

\begin{lem}\label{second_lemma_of_third_theorem}
For $H=H_{1, 5}$, there exists an embedding from $A(\Gamma_{0})$ to ${\rm Mod}(H)$, but $\Gamma_{0}$ does not embed into $\mathcal{D}(H)$ as an induced subgraph.
\end{lem}

\begin{proof}[Proof of Lemma~\ref{second_lemma_of_third_theorem}]
First, we show the first half of the conclusion.
By Lemma~\ref{first_lemma_of_third_thm_Kim-Koberda13}, $A(\Gamma_{0})\leq A(\Gamma_{1})$.
By Lemma~\ref{embedding_Gamma_{1}_into_disk_graphs} and Lemma~\ref{lemma_of_main_thm}, $A(\Gamma_{1})\leq {\rm Mod}(H)$.
Therefore, it follows that $A(\Gamma_{0})\leq {\rm Mod}(H)$.
For the second half, we suppose that $\Gamma_{0}\leq \mathcal{D}(H)$.
We regard $a$, $b$, $c$, $d$, $g$, $h$, and $q$ as disks in $H$.
From $C_{4}\leq \Gamma_{0}$, we have one of the seventeen cases in Lemma~\ref{17_cases _of_complexity_five_handlebodies}.
By the definition of disk graphs, the intersections $q\cap g$, $q\cap h$, $g\cap S_{2}$, $h\cap S_{1}$, and $g\cap h$ are not empty.
Moreover $\partial q\subseteq S_{0}$ and $g\cap S_{1}=h\cap S_{2}=\emptyset$.

In Case (1), the annulus $S_{0}$ connects $S_{1}$ and $S_{2}$.
This implies that $\partial g\subseteq S_{2}$ and $\partial h\subseteq S_{1}$, and so $g\cap h=\emptyset$.
This is a contradiction.
In Case (2), since $\partial q\subseteq S_{0}\approx S_{0, 3}$, we have $\partial q=S_{0}\cap S_{1}$ or $\partial q=S_{0}\cap S_{2}$.
Without loss of generality, we may assume $\partial q=S_{0}\cap S_{1}$.
The curve $\partial q$ is a separating curve which separates $S_{1}$ from $\partial H$.
By the fact that $g\cap S_{1}=\emptyset$, we have $g\cap q=\emptyset$.
This contradicts the fact that $g\cap q\not=\emptyset$.
In Cases (3), (4), and (5), if $\partial g$ intersects $\partial h$, then they intersect on the components of $S_{0}$ which connect $S_{1}$ and $S_{2}$.
Similarly to Case (1), we have $g\cap h=\emptyset$ and this is a contradiction.
Case (6) does not appear for $H_{1, 5}$.
Note that we see $\Gamma_{0}$ is embedded in $\mathcal{D}(H_{0, 8})$.
In Case (7), by a similar argument to that of Case (1) we have $g\cap h=\emptyset$, and this is a contradiction.
In Case (8), we have $\partial q\subseteq S_{0}\cap S_{1}$ or $\partial q=S_{0}\cap S_{2}$ since $\partial q\subseteq S_{0}\approx S_{0, 3}$.
Without loss of generality we may assume that $\partial q\subseteq S_{0}\cap S_{1}$.
By the fact that $g\cap q\not=\emptyset$, $\partial g$ intersects $\partial q$.
Then it follows that $\partial g$ intersects $S_{1}$.
This is a contradiction.
In Case (9), by a similar argument to that of Case (8), we have $g\cap q\not=\emptyset$.
This is a contradiction.
In Case (11), by a similar argument to that of Case (2), we have $g\cap q=\emptyset$.
This contradicts the fact that $g\cap q\not=\emptyset$.
In Cases (10), (12), (13), (14), (15), (16), and (17), by a similar argument to that of Case (3), we have $g\cap h=\emptyset$.
This is a contradiction.
Therefore it follows that $\Gamma_{0}\not\leq \mathcal{D}(H_{1, 5})$.
\end{proof}

\begin{lem}{\rm(}\cite[Lemma 14]{Kim-Koberda13}{\rm)}\label{second_lemma_of_third_theorem_Kim-Koberda13}
Let $S$ be a surface with $\xi(S)=4$.
There exists an embedding from $A(\Gamma_{0})$ to ${\rm Mod}(S)$, but $\Gamma_{0}$ does not embed into $\mathcal{C}(S)$ as an induced subgraph.
\end{lem}

\begin{proof}[Proof of Lemma~\ref{second_lemma_of_third_theorem_Kim-Koberda13}]
For the proof see the proof of \cite[Lemma 14]{Kim-Koberda13}.
\end{proof}

\begin{lem}\label{first_lemma_of_third_theorem}
For $H=H_{0, 7}$, there exists an embedding from $A(\Gamma_{0})$ to ${\rm Mod}(H)$, but $\Gamma_{0}$ does not embed into $\mathcal{D}(H)$ as an induced subgraph.
\end{lem}

\begin{proof}[Proof of Lemma~\ref{first_lemma_of_third_theorem}]
It directly follows from Lemma~\ref{second_lemma_of_third_theorem_Kim-Koberda13}, since the handlebody group of $H$ is isomorphic to the mapping class group of $\partial H$, and the disk graph of $H$ is isomorphic to the curve graph of $\partial H$.
\end{proof}

From Lemmas~\ref{second_lemma_of_third_theorem} and \ref{first_lemma_of_third_theorem}, we have completed the proof of Theorem~\ref{third_thm}.
Therefore the converse of Theorem~\ref{main_thm} is not true in general.

\begin{rem}\label{Notes for high complexity cases of handlebody groups}
For a graph $\Gamma$, we define $\eta(\Gamma)$ to be the minimum of $\xi(H)$ among connected handlebodies $H$ satisfying $\Gamma\leq \mathcal{D}(H)$. 
From Lemma~\ref{second_lemma_of_third_theorem_Kim-Koberda13}, we see $\eta(\Gamma_{0})>4$.
The graph $\Gamma_{0}$ embeds into $\mathcal{D}(H_{0, 8})$ (see Figure~\ref{fig_complexity_five_handlebody_embedded_gamma0_v2}, the handlebody of Case (6) in Lemma~\ref{17_cases _of_complexity_five_handlebodies}).
Hence we see $\eta(\Gamma_{0})=5$.
Further, there exists a handle body $H$ whose complexity is four such that $\Gamma_{1}$ cannot embed into $\mathcal{D}(H)$.
From these facts, we cannot apply the arguments of Kim-Koberda~\cite[Section 4.3]{Kim-Koberda13} for high complexity cases to handlebodies.
However, if $H$ is a handlebody whose complexity is $n\geq 4$ and $H_{0, 7}\subseteq H$, then there exists a graph $\Lambda_{n}$ such that $A(\Lambda_{n})\leq {\rm Mod}(H)$ but $\Lambda_{n}\not\leq \mathcal{D}(H)$ by the same argument as that of the proof of \cite[Proposition 16]{Kim-Koberda13}.
\end{rem}

\begin{figure}[h]
\includegraphics[scale=0.55]{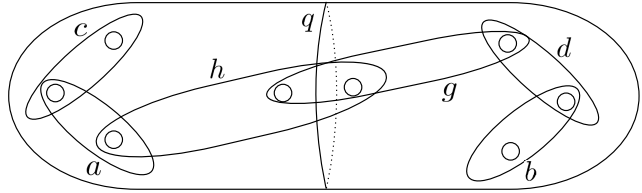}
\caption{The graph $\Gamma_{0}$ is embedded into $\mathcal{D}(H_{0, 8})$.}\label{fig_complexity_five_handlebody_embedded_gamma0_v2}
\end{figure}

%%%%%%%%%%%%%%%%%
%%%%%%%%%%%%%%%%%
%%%%%%%%%%%%%%%%%
%%%Proof of fourth theorem%%
%%%%%%%%%%%%%%%%%
%%%%%%%%%%%%%%%%%
%%%%%%%%%%%%%%%%%
\section{Proof of Theorem~\ref{fourth_thm}}

If $A$ is a multi-disk on $H$, then we denote by $\langle A\rangle$ a subgroup of ${\rm Mod}(H)$ which is generated by disk twists along the disks in $A$.
The proof of Theorem~\ref{fourth_thm} comes from the proof of Kim-Koberda~\cite[Theorem 5]{Kim-Koberda13}.

\begin{proof}[Proof of Theorem~\ref{fourth_thm}]
Let $v$ be an arbitrary vertex of $V(\Gamma)$.
We write $K$ and $L$ for two disjoint cliques of $\Gamma$ such that $K\coprod\{v\}$ and $L\coprod\{v\}$ are cliques on $N$ vertices of $\Gamma$.
We note that there are such cliques in $\Gamma$ for any $v\in V(\Gamma)$ because of the assumption.
The support of $f \langle K\rangle$ is a regular neighborhood of a multi-disk in $H$, and we call the multi-disk $A$.
Similarly, we write $B$ and $C$ for multi-disks in the supports of $f \langle L\rangle$ and $f \langle v\rangle$ in $H$ respectively.
Since $\xi(H)=N$, multi-disks $A\cup C$ and $B\cup C$ are maximal.
Note that $\langle C\rangle$ is a subgroup of $\langle A\cup C\rangle\cap\langle B\cup C\rangle$.
By the diagram in Figure~\ref{fig_diagram_v2}, $f \langle v\rangle$ is a finite index subgroup of $\langle C\rangle$.
By the fact that $\langle C\rangle\cong \mathbb{Z}^{|C|}$ and $f \langle v\rangle\cong \mathbb{Z}$, it follows that $|C|=1$.
Hence, for each $v\in V(\Gamma)$, $f(v)$ is some single disk twist $\delta_{d}$ along a disk $d$, and we obtain an injection from $\Gamma$ into $\mathcal{D}(H)$.
\end{proof}

\begin{figure}[h]
\includegraphics[scale=0.45]{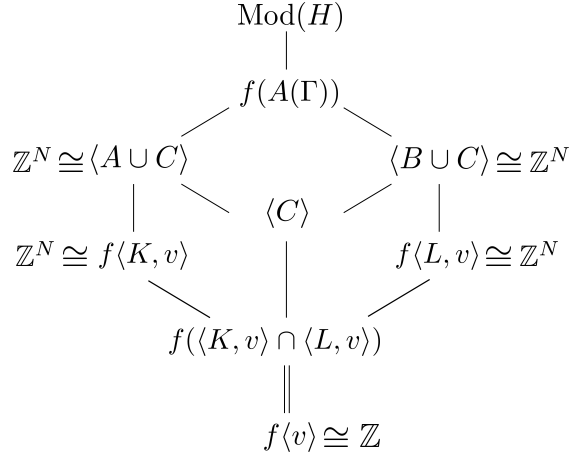}
\caption{A figure representing a relationship between groups and their subgroups.}\label{fig_diagram_v2}
\end{figure}

%%%%%%%%%%%%%%%%%
%%%%%%%%%%%%%%%%%
%%%%%%%%%%%%%%%%%
%%%%%%Appendix%%%%%%
%%%%%%%%%%%%%%%%%
%%%%%%%%%%%%%%%%%
%%%%%%%%%%%%%%%%%
\section{Appendix}\label{appendix}

In this section, we prove the following proposition.

\begin{prop}\label{last_proposition}
The graph $\Gamma_{1}$ is not embedded into the disk graphs of $H$ with $\xi(H)\leq 3$, $H=H_{1, 4}$, $H=H_{2, 1}$, and $H=H_{2, 2}$.
\end{prop}

\begin{proof}[Proof of Proposition~\ref{last_proposition}]
First $\Gamma_{1}$ is not embedded into the disk graphs of $H$ with $\xi(H)\leq 3$ because $\Gamma_{1}$ has cliques on four vertices.
Next we suppose that $H$ is a handlebody with $\xi(H)=4$, that is, $H=H_{0, 7}$, $H=H_{1, 4}$, or $H=H_{2, 1}$.
We will show that $\Gamma_{1}$ is embedded into only $\mathcal{D}(H_{0, 7})$.
We suppose that $\Gamma_{1}\leq \mathcal{D}(H)$.
By the fact that $C_{4}\leq \Gamma_{1}$, we obtain one of the five cases in Lemma~\ref{five_cases_of complexity_four_handlebodies}.
From the definition of disk graphs, we see $e\cap g=\emptyset$, $f\cap h=\emptyset$, $e\cap f=\emptyset$, $g\cap S_{1}=\emptyset$, and $h\cap S_{2}=\emptyset$.
Further, $g\cap h\not=\emptyset$, $g\cap f\not=\emptyset$, $h\cap e\not=\emptyset$, $e\subseteq S_{0}$, and $f\subseteq S_{0}$.
In Case (1)$'$, the annulus $S_{0}$ connect $S_{1}$ and $S_{2}$.
This implies that $g\subseteq S_{2}$ and $h\subseteq S_{1}$, and so $g\cap h=\emptyset$.
This is a contradiction.
In Case (2)$'$, $S_{1}$ and $S_{2}$ are homeomorphic to $S_{0, 4}$.
One can confirm that $\Gamma_{1}$ is embedded into only $\mathcal{D}(H_{0, 7})$ in this case.
Note that if we discuss it for a surface $S$ with $\xi(S)=4$, then $S_{1}$ and $S_{2}$ are homeomorphic to $S_{0, 4}$ or $S_{1, 1}$, and so $\Gamma_{1}$ is embedded into $\mathcal{C}(S_{0, 7})$, $\mathcal{C}(S_{1, 4})$, and $\mathcal{C}(S_{2, 1})$.
In Cases (3)$'$, (4)$'$, and (5)$'$, by the same argument as that of Case (1)$'$, we have $g\cap h=\emptyset$, and this is a contradiction.

Secondly we suppose that $H$ is a handlebody with $\xi(H)=5$, that is, $H=H_{0, 8}$, $H=H_{1, 5}$, or $H=H_{2, 2}$.
We will show that $\Gamma_{1}$ is embedded into only $\mathcal{D}(H_{0, 8})$ and $\mathcal{D}(H_{1, 5})$.
In Cases (1), (3), (4), (5), (7), (10), (11), (12), (13), (14), (15), (16), and (17), by the same argument that $\Gamma_{0}$ is not embedded into $\mathcal{D}(H_{1, 5})$ in the proof of Lemma~\ref{second_lemma_of_third_theorem}, we have $g\cap h=\emptyset$, and this is a contradiction.
In Case (2), we see $\Gamma_{1}$ is embedded into only $\mathcal{D}(H_{0, 8})$ and $\mathcal{D}(H_{1, 5})$.
In Case (6), we see $\Gamma_{1}$ is embedded into only $\mathcal{D}(H_{0, 8})$.
In Cases (8) and (9), we see $\Gamma_{1}$ is embedded into only $\mathcal{D}(H_{1, 5})$.
By the argument above, we have finished the proof of the proposition. 
\end{proof}

%%%%%%%%%%%%%%%%%
%%%%%%%%%%%%%%%%%
%%%%%%%%%%%%%%%%%
%%%Acknowledgements%%%%
%%%%%%%%%%%%%%%%%
%%%%%%%%%%%%%%%%%
%%%%%%%%%%%%%%%%%
\par
{\bf Acknowledgements:} The author is deeply grateful to Hisaaki Endo for his warm encouragement and helpful advice.
The author also wishes to thank Susumu Hirose and Yeonhee Jang for their meticulous comments and pieces of advice about handlebody groups and disk complexes.

%%%%%%%%%%%%%%%%%%%
%%%%%%%%%%%%%%%%%%%
%%%%%%%%%%%%%%%%%%%
%%%Reference%%%%%%%%%%%
%%%%%%%%%%%%%%%%%%%
%%%%%%%%%%%%%%%%%%%
%%%%%%%%%%%%%%%%%%%

\end{document}